\documentclass[11pt]{article}

\usepackage{amssymb,amsmath,amsfonts,amsthm}
\usepackage{latexsym}
\usepackage{graphics}
\usepackage{indentfirst}
\usepackage{hyperref}
\usepackage{comment}
\usepackage{bm}
\allowdisplaybreaks

\newtheorem*{thm*}{Theorem}
\newtheorem{thm}{Theorem}
\newtheorem{lem}[thm]{Lemma}

\newtheorem{ques}[thm]{Question}

\newcommand{\N}{\mathbb{N}}

\newcommand{\HH}{\mathcal{H}}
\newcommand{\II}{\mathcal{I}}

\begin{document}

\title{Non-chromatic-adherence of the DP Color Function via Generalized Theta Graphs}

\author{Manh Vu Bui$^1$, Hemanshu Kaul$^2$, Michael Maxfield$^1$, Jeffrey A. Mudrock$^1$, \\ Paul Shin$^1$, and Seth Thomason$^1$}

\footnotetext[1]{Department of Mathematics, College of Lake County, Grayslake, IL 60030.  E-mail:  {\tt {jmudrock@clcillinois.edu}}}

\footnotetext[2]{Department of Applied Mathematics, Illinois Institute of Technology, Chicago, IL 60616. E-mail: {\tt {kaul@iit.edu}}}
\maketitle

\begin{abstract}
DP-coloring (also called correspondence coloring) is a generalization of list coloring that has been widely studied in recent years after its introduction by Dvo\v{r}\'{a}k and Postle in 2015. The chromatic polynomial of a graph is an extensively studied notion in combinatorics since its introduction by Birkhoff in 1912; denoted $P(G,m)$, it equals the number of proper $m$-colorings of graph $G$. Counting function analogues of the chromatic polynomial have been introduced and studied for list colorings:  $P_{\ell}$, the list color function (1990); DP colorings: $P_{DP}$, the DP color function (2019), and $P^*_{DP}$, the dual DP color function (2021). For any graph $G$ and $m \in \N$, $P_{DP}(G, m) \leq P_\ell(G,m) \leq P(G,m) \leq P_{DP}^*(G,m)$.  A function $f$ is chromatic-adherent if for every graph $G$, $f(G,a) = P(G,a)$ for some $a \geq \chi(G)$ implies that $f(G,m) = P(G,m)$ for all $m \geq a$. It is not known if the list color function and the DP color function are chromatic-adherent. We show that the DP color function is not chromatic-adherent by studying the DP color function of Generalized Theta graphs. The tools we develop along with the Rearrangement Inequality give a new method for determining the DP color function of all Theta graphs and the dual DP color function of all Generalized Theta graphs.
\medskip




\end{abstract}

\section{Introduction}\label{intro}

In this paper all graphs are nonempty, finite, simple graphs unless otherwise noted.  Generally speaking we follow West~\cite{W01} for terminology and notation.  The set of natural numbers is $\N = \{1,2,3, \ldots \}$.  For $m \in \N$, we write $[m]$ for the set $\{1, \ldots, m \}$.  We write \emph{AM-GM Inequality} for the inequality of arithmetic and geometric means.  Given a set $A$, $\mathcal{P}(A)$ is the power set of $A$.  If $G$ is a graph and $S, U \subseteq V(G)$, we use $G[S]$ for the subgraph of $G$ induced by $S$, and we use $E_G(S, U)$ to denote the subset of $E(G[S \cup U])$ with at least one endpoint in $S$ and at least one endpoint in $U$.  If $u$ and $v$ are adjacent in $G$, $uv$ or $vu$ refers to the edge between $u$ and $v$.  If $e \in E(G)$, we write $G \cdot e$ for the graph obtained from $G$ by contracting the edge $e$.  

In this paper, we will use the following version of the Rearrangement Inequality (\cite{R52}) several times.
\begin{thm} [Rearrangement Inequality] \label{thm: RI}
For $k,n \in \N$, suppose that for each $(i,j) \in [k] \times [n]$, $x_{(i,j)}$ is a nonnegative real number so that $x_{(i,1)} \leq \cdots \leq x_{(i,n)}$ for each $i \in [k]$.  Let $\sigma_i$ be an arbitrary permutation of $[n]$ for each $i \in [k]$.  If $k=2$, then 
$$\sum_{j=1}^n x_{(1, n+1-j)} x_{(2,j)} \leq \sum_{j=1}^n x_{(1, \sigma_1(j))} x_{(2,\sigma_2(j))}.$$ 
Furthermore, for any $k \in \N$,
$ \sum_{j=1}^n \prod_{i=1}^k x_{(i, \sigma_i(j))} \leq \sum_{j=1}^n \prod_{i=1}^k x_{(i, j)}.$
\end{thm}
Importantly, the first inequality in Theorem~\ref{thm: RI} doesn't hold when $k \geq 3$ (see e.g.,~\cite{W20}).  Throughout this paper whenever $l < j$, we take $\prod_{i=j}^l a_i$ to equal 1.

\subsection{DP-Coloring} \label{basic}

In the classical vertex coloring problem we wish to color the vertices of a graph $G$ with up to $m$ colors from $[m]$ so that adjacent vertices receive different colors, a so-called \emph{proper $m$-coloring}. The chromatic number of a graph $G$, denoted $\chi(G)$, is the smallest $m$ such that $G$ has a proper $m$-coloring.  List coloring is a well-known variation on classical vertex coloring which was introduced independently by Vizing~\cite{V76} and Erd\H{o}s, Rubin, and Taylor~\cite{ET79} in the 1970s.  For list coloring, we associate a \emph{list assignment} $L$ with a graph $G$ such that each vertex $v \in V(G)$ is assigned a list of available colors $L(v)$ (we say $L$ is a list assignment for $G$).  Then, $G$ is \emph{$L$-colorable} if there exists a proper coloring $f$ of $G$ such that $f(v) \in L(v)$ for each $v \in V(G)$ (we refer to $f$ as a \emph{proper $L$-coloring} of $G$).  A list assignment $L$ is called a \emph{$k$-assignment} for $G$ if $|L(v)|=k$ for each $v \in V(G)$.  The \emph{list chromatic number} of a graph $G$, denoted $\chi_\ell(G)$, is the smallest $k$ such that $G$ is $L$-colorable whenever $L$ is a $k$-assignment for $G$.  We say $G$ is \emph{$k$-choosable} if $k \geq \chi_\ell(G)$.  Note $\chi(G) \leq \chi_\ell(G)$, and this inequality may be strict since it is known that there are bipartite graphs with arbitrarily large list chromatic number (see~\cite{ET79}).

In 2015, Dvo\v{r}\'{a}k and Postle~\cite{DP15} introduced a generalization of list coloring called DP-coloring (they called it correspondence coloring) in order to prove that every planar graph without cycles of lengths 4 to 8 is 3-choosable. DP-coloring has been extensively studied  over the past 6 years (see e.g.,~\cite{B17, BK182, KM20, KO18, LL19, LLYY19, Mo18, M18}).  Intuitively, DP-coloring is a variation on list coloring where each vertex in the graph still gets a list of colors, but identification of which colors are different can change from edge to edge.  Following~\cite{BK17}, we now give the formal definition. Suppose $G$ is a graph.   A \emph{cover} of $G$ is a pair $\mathcal{H} = (L,H)$ consisting of a graph $H$ and a function $L: V(G) \rightarrow \mathcal{P}(V(H))$ satisfying the following four requirements:

\vspace{5mm}

\noindent(1) the set $\{L(u) : u \in V(G) \}$ is a partition of $V(H)$ of size $|V(G)|$; \\
(2) for every $u \in V(G)$, the graph $H[L(u)]$ is complete; \\
(3) if $E_H(L(u),L(v))$ is nonempty, then $u=v$ or $uv \in E(G)$; \\
(4) if $uv \in E(G)$, then $E_H(L(u),L(v))$ is a matching (the matching may be empty).

\vspace{5mm}

Suppose $\mathcal{H} = (L,H)$ is a cover of $G$.  We refer to the edges of $H$ connecting distinct parts of the partition $\{L(u) : u \in V(G) \}$ as \emph{cross-edges}.  An \emph{$\mathcal{H}$-coloring} of $G$ is an independent set in $H$ of size $|V(G)|$.  It is immediately clear that an independent set $I \subseteq V(H)$ is an $\mathcal{H}$-coloring of $G$ if and only if $|I \cap L(u)|=1$ for each $u \in V(G)$.  We say $\mathcal{H}$ is \emph{$m$-fold} if $|L(u)|=m$ for each $u \in V(G)$.  An $m$-fold cover $\mathcal{H}$ is a \emph{full cover} if for each $uv \in E(G)$, the matching $E_{H}(L(u),L(v))$ is perfect.  The \emph{DP-chromatic number} of $G$, $\chi_{DP}(G)$, is the smallest $m \in \N$ such that $G$ has an $\mathcal{H}$-coloring whenever $\mathcal{H}$ is an $m$-fold cover of $G$.

Suppose $\mathcal{H} = (L,H)$ is an $m$-fold cover of $G$.  We say that $\mathcal{H}$ has a \emph{canonical labeling} if it is possible to name the vertices of $H$ so that $L(u) = \{ (u,j) : j \in [m] \}$ and $(u,j)(v,j) \in E(H)$ for each $j \in [m]$ whenever $uv \in E(G)$.~\footnote{When $\mathcal{H}=(L,H)$ has a canonical labeling, we will always refer to the vertices of $H$ using this naming scheme.}  Now, suppose $\mathcal{H}$ has a canonical labeling and $G$ has a proper $m$-coloring.  Then, if $\mathcal{I}$ is the set of $\mathcal{H}$-colorings of $G$ and $\mathcal{C}$ is the set of proper $m$-colorings of $G$, the function $f: \mathcal{C} \rightarrow \mathcal{I}$ given by $f(c) = \{ (v, c(v)) : v \in V(G) \}$ is a bijection.  Also, given an $m$-assignment $L$ for a graph $G$, it is easy to construct an $m$-fold cover $\mathcal{H}'$ of $G$ such that $G$ has an $\mathcal{H}'$-coloring if and only if $G$ has a proper $L$-coloring (see~\cite{BK17}).  So, $\chi(G) \leq \chi_\ell(G) \leq \chi_{DP}(G)$.
 
\subsection{The DP Color Function and Dual DP Color Function}

In 1912, Birkhoff introduced the chromatic polynomial of a graph in hopes of using it to make progress on the four color problem.  For $m \in \N$, the \emph{chromatic polynomial} of a graph $G$, $P(G,m)$, is the number of proper $m$-colorings of $G$. It is well-known that $P(G,m)$ is a polynomial in $m$ of degree $|V(G)|$ (see~\cite{B12}).  For example, $P(K_n,m) = \prod_{i=0}^{n-1} (m-i)$, $P(C_n,m) = (m-1)^n + (-1)^n (m-1)$ whenever~\footnote{When considering graphs with multiple edges or loops, one should note that this formula for the chromatic polynomial of a cycle also works for $C_1$ and $C_2$.} $n \geq 3$, and $P(T,m) = m(m-1)^{n-1}$ whenever $T$ is a tree on $n$ vertices (see~\cite{W01}).  

The notion of chromatic polynomial was extended to list coloring in the early 1990s. If $L$ is a list assignment for $G$, we use $P(G,L)$ to denote the number of proper $L$-colorings of $G$. The \emph{list color function} $P_\ell(G,m)$ is the minimum value of $P(G,L)$ where the minimum is taken over all possible $m$-assignments $L$ for $G$.  It is clear that $P_\ell(G,m) \leq P(G,m)$ for each $m \in \N$ since we must consider the $m$-assignment that assigns the same $m$ colors to all the vertices in $G$ when considering all possible $m$-assignments for $G$.  In general, the list color function can differ significantly from the chromatic polynomial for small values of $m$.  However, for large values of $m$, Wang, Qian, and Yan~\cite{WQ17} (improving upon results in~\cite{D92} and~\cite{T09}) showed:
If $G$ is a connected graph with $l$ edges, then $P_{\ell}(G,m)=P(G,m)$ whenever $m > (l-1)/\ln(1+ \sqrt{2})$.  It is also known that $P_{\ell}(G,m)=P(G,m)$ for all $m \in \N$ when $G$ is a cycle or chordal (see~\cite{KN16} and~\cite{AS90}).  

A fundamental open question on the list color function asks whether the list color function of a graph and the corresponding chromatic polynomial stay the same after the first point at which they are both nonzero and equal.  Let $\mathcal{G}$ be the set of all finite, simple graphs.  We say a function $f: \mathcal{G} \times \mathbb{N} \rightarrow \mathbb{N}$ is \emph{chromatic-adherent} if for every graph $G$,  $f(G,a) = P(G,a)$ for some $a \geq \chi(G)$ implies that $f(G,m) = P(G,m)$ for all $m \geq a$. 

\begin{ques} [\cite{KN16}] \label{ques: funlist}
Is $P_{\ell}$ chromatic-adherent?
\end{ques}

In 2019, two of the authors (Kaul and Mudrock in~\cite{KM19}) introduced a DP-coloring analogue of the chromatic polynomial in hopes of gaining a better understanding of DP-coloring and using it as a tool for making progress on some open questions related to the list color function~\cite{KM19}.  Since its introduction in 2019, the DP color function has received some attention in the literature (see e.g.,~\cite{BH21, DY21, HK21, KM21, M21, MT20}).   Suppose $\mathcal{H} = (L,H)$ is a cover of graph $G$.  Let $P_{DP}(G, \mathcal{H})$ be the number of $\mathcal{H}$-colorings of $G$.  Then, the \emph{DP color function} of $G$, $P_{DP}(G,m)$, is the minimum value of $P_{DP}(G, \mathcal{H})$ where the minimum is taken over all possible $m$-fold covers $\mathcal{H}$ of $G$.  A similar tool for studying enumerative aspects of DP coloring was recently introduced~\cite{M21}; specifically, the \emph{dual DP color function} of $G$, denoted $P^*_{DP}(G,m)$, is the maximum value of $P_{DP}(G, \mathcal{H})$ where the maximum is taken over all full $m$-fold covers $\mathcal{H}$ of $G$.~\footnote{We take $\N$ to be the domain of the DP color function and dual DP color function of any graph.} It is easy to show that for any graph $G$ and $m \in \N$, 
$$P_{DP}(G, m) \leq P_\ell(G,m) \leq P(G,m) \leq P_{DP}^*(G,m).$$
Note that if $G$ is a disconnected graph with components: $H_1, H_2, \ldots, H_t$, then $P_{DP}(G, m) = \prod_{i=1}^t P_{DP}(H_i,m)$ (an analogous result holds for the dual DP color function).  So, we will only consider connected graphs from this point forward unless otherwise noted.

The list color function and DP color function of certain graphs behave similarly.  However, for some graphs there are surprising differences.  For example, similar to the list color function,  $P_{DP}(G,m) = P(G,m)$ for every $m \in \N$  whenever $G$ is chordal or an odd cycle~\cite{KM19}.  On the other hand, unlike the list color function, it is well-known that $P_{DP}(G,m)$ does not necessarily equal $P(G,m)$ for sufficiently large $m$.  Indeed Dong and Yang~\cite{DY21} recently generalized a result of Kaul and Mudrock~\cite{KM19} and showed that if $G$ is a simple graph that contains an edge $e$ such that the length of a shortest cycle containing $e$ is even, then there exists an $N \in \N$ such that $P_{DP}(G,m) < P(G,m)$ whenever $m \geq N$.  Nevertheless, while introducing the DP color function, Kaul and Mudrock asked the analogue of Question~\ref{ques: funlist} for the DP color function.
\begin{ques} [\cite{KM19}] \label{ques: funDP}
Is $P_{DP}$ chromatic-adherent?
\end{ques}
We will see below that the answer to Question~\ref{ques: funDP} is no.

\subsection{Summary of Results}

We answer Question~\ref{ques: funDP} by studying the DP color function of Generalized Theta graphs.   A \emph{Generalized Theta graph} $\Theta(l_1, \ldots, l_n)$ consists of a pair of end vertices joined by $n$ internally disjoint paths of lengths $l_1, \ldots, l_n \in \N$. When $n=3$, $\Theta(l_1, l_2, l_3)$ is simply called a \emph{Theta graph}.  From this point forward, we will always assume that the lengths of the paths of a Generalized Theta graph are listed such that: if $l_1, \ldots, l_n$ don't all have the same parity, then there exists some $r \in [n] - \{1\}$ such that $l_2, \ldots, l_r$ have the same parity which is different from that of $l_1, l_{r + 1}, \ldots, l_n$.  

It is easy to prove that $\chi_{DP}(\Theta(l_1, \ldots, l_n))=3$ whenever $n \geq 2$.  It is also well-known (see~\cite{BH01}) that if $G = \Theta(l_1, \ldots, l_n)$, then for each $m \geq 2$, 
\[P(G,m) = \frac{\prod_{i=1}^{n}((m-1)^{l_i+1}+(-1)^{l_i+1}(m-1))}{(m(m-1))^{n-1}} + \frac{\prod_{i=1}^n((m-1)^{l_i}+(-1)^{l_i}(m-1))}{m^{n-1}}.\] 
Generalized Theta graphs have been widely studied for many graph theoretic problems (see e.g.,~\cite{BF16, CM15, ET79, LB16, LB19, MM21, SJ18}), and they are the main subject of two classical papers on the chromatic polynomial~\cite{BH01} and~\cite{S04} which include the celebrated result that the zeros of the chromatic polynomials of the Generalized Theta graphs are dense in the whole complex plane with the possible exception of the unit disc around the origin (by including the join of Generalized Theta graphs with $K_2$ this extends to all of the complex plane). Recently, exact formulas for the DP color function of all Theta graphs were determined, and it was shown that when $G = \Theta(l_1, \ldots, l_n)$, there is a polynomial $p(m)$ and $N \in \N$ such that $P_{DP}(G,m) = p(m)$ whenever $m \geq N$ (see~\cite{HK21}).

In Section~\ref{counter} we develop some elementary tools for analyzing the DP color function of a Generalized Theta graph.  These tools allow us to establish a sufficient condition for $P_{DP}(G,m)=P(G,m)$ when $G$ is a Generalized Theta graph.  This sufficient condition ultimately allows us to find two examples of graphs that demonstrate the answer to Question~\ref{ques: funDP} is no.
\begin{thm} \label{thm: negative}
If $G$ is $\Theta(2,3,3,3,2)$ or $\Theta(2,3,3,3,3,3,2,2)$, then $P_{DP}(G,3)=P(G,3)$ and there is an $N \in \N$ such that $P_{DP}(G,m) < P(G,m)$ for all $m \geq N$. 
\end{thm}
Interestingly, Theorem~\ref{thm: negative} contains the only examples that we know of that demonstrate that the answer to Question~\ref{ques: funDP} is no.  So, the following question is natural.
\begin{ques} \label{ques: open}
For which graphs $G$ do there exist, $a,b \in \N$ with $\chi(G) \leq a < b$, $P_{DP}(G,a) = P(G,a)$, and $P_{DP}(G,b) < P(G,b)$? 
\end{ques}
We do not even know whether there are infinitely many Generalized Theta graphs that satisfy the conditions of Question~\ref{ques: open}.  Since little is known about the enumerative aspects of DP coloring Generalized Theta graphs, we continue by studying the DP color function and dual DP color function of Generalized Theta graphs.

In Section~\ref{theta} we show how the tools we developed in Section~\ref{counter} along with the Rearrangement Inequality give us an elementary way to derive the formulas for the DP color function of all Theta graphs.
\begin{thm}\label{thm: generalTheta3}
Let $G = \Theta(l_1,l_2,l_3)$, where $l_1 = \min_{i \in [3]} l_i \geq 1$ and $l_i \geq 2$ for each $i \in \{2, 3\}$.

(i) If the parity of $l_1$ is different from both $l_2$ and $l_3$, then $P_{DP}(G, m) = P(G, m)$ for all $m \in \mathbb{N}$.

(ii) If the parity of $l_1$ is the same as $l_3$ and different from $l_2$, then for $m \geq 2$:

\noindent $\displaystyle P_{DP}(G, m) = \frac{1}{m} \left[ (m - 1)^{l_1 + l_2 + l_3} + (m - 1)^{l_1} - (m - 1)^{l_2} - (m - 1)^{l_3 + 1} + (-1)^{l_2 + 1}(m - 2) \right]$.

(iii) If $l_1$, $l_2$, and $l_3$ all have the same parity, then for $m \geq 3$:

\noindent $\displaystyle P_{DP}(G, m) = \frac{1}{m} \left[ (m - 1)^{l_1 + l_2 + l_3} - (m - 1)^{l_1} - (m - 1)^{l_2} - (m - 1)^{l_3} + 2(-1)^{l_1 + l_2 + l_3} \right]$.
\end{thm}

Finally, in Section~\ref{dual} we build on some of the ideas in Section~\ref{theta} and completely determine the dual DP color function of all Generalized Theta graphs.

\section{The DP Color Function is not Chromatic-adherent} \label{counter} 

We begin this Section by establishing some conventions and notation that will be used for the remainder of this paper.  Whenever $\HH = (L,H)$ is an $m$-fold cover of $G$ and $P \subseteq V(H)$, we let $N(P,\HH)$ be the number of $\HH$-colorings containing $P \subseteq V(H)$.  Suppose $G=\Theta(l_1,\ldots,l_n)$.  We will always assume $n \geq 2$, $l_1 = \min_{i \in [n]} l_i \geq 1$, and $l_i \geq 2$ for each $i \in [n] - \{1\}$.  The endpoints of the paths that make up $G$ will always be called $u$ and $w$. When $\HH = (L,H)$ is a full $m$-fold cover of $G$, we always suppose $L(v) = \{(v, j) : j \in [m]\}$ for each $v \in V(G)$. Furthermore, we let $G_i$ be the $u,w$-path of length $l_i$ used to form $G$ and $R_i = V(G_i)$. We let $\HH_i = (L_i,H_i)$ where $L_i$ is $L$ with domain restricted to $R_i$ and $H_i$ is the graph defined by $H[\bigcup_{v \in R_i} L(v)] - E_H(L(u), L(w))$ if $l_1 = 1$ and $i \neq 1$, and $H[\bigcup_{v \in R_i} L(v)]$ otherwise. It is easy to see that $\HH_i$ is a full $m$-fold cover of $G_i$.

We are now ready to present two lemmas that will be of fundamental importance throughout the paper.
\begin{lem}\label{lem: genThetaCount}
Let $G=\Theta(l_1,\ldots,l_n)$. Let $\HH = (L,H)$ be a full $m$-fold cover of $G$ where $m \geq 2$. Then $$P_{DP}(G,\HH) = \sum_{(i,j) \in [m]^2} \prod_{k=1}^{n} N(\{(u,i),(w,j)\},\HH_k).$$
\end{lem}

\begin{proof}
Notice $P_{DP}(G,\HH) = \sum_{(i,j) \in [m]^2} N(\{(u,i),(w,j)\},\HH)$. Let $A = \{(u,j_1),(w,j_2)\}$ where $j_1,j_2$ are fixed elements of $[m]$. We will show that $N(A,\HH) = \prod_{i = 1}^n N(A,\HH_i)$. Notice that if $(u,j_1)(w,j_2) \in E(H)$, then $N(A,\HH) = N(A,\HH_1) = 0$. So we can assume that $(u,j_1)(w,j_2) \notin E(H)$. Let $\II_i$ be the set of all $\HH_i$-colorings of $G_i$ that contain $A$, and let $\II$ be the set of all $\HH$-colorings of $G$ that contain $A$. If $\mathcal{I}_k$ is empty for some $k \in [n]$, then it is easy to see that $\mathcal{I}$ must also be empty; hence, $\prod_{i = 1}^n N(A, \mathcal{H}_i) = \prod_{i = 1}^n \lvert \mathcal{I}_i \rvert = 0 = \lvert \mathcal{I} \rvert = N(A, \mathcal{H})$. So, suppose that $\mathcal{I}_i$ is nonempty for each $i \in [n]$. Let $f: \prod_{i=1}^{n} \II_i \rightarrow \II$ be the function given by $f((I_1,\ldots,I_n)) = \bigcup_{i=1}^{n} I_i$. It is easy to check that $\bigcup_{i=1}^{n} I_i$ is an independent set of size $|V(G)|$ in $H$. Clearly, $f$ is a bijection. As such, $N(A,\HH) = |\II| = \prod_{i=1}^{n} |\II_i| = \prod_{i=1}^{n} N(A,\HH_i)$.
\end{proof}

\begin{lem}\label{lem: binaryThetaValues}
Let $G = \Theta(l_1,\ldots,l_n)$. Let $\HH = (L,H)$ be a full $m$-fold cover of $G$ where $m \geq 2$. Let $(i,j) \in [m]^2$. For $k \in [n]$, if there is a path in $H_k$ connecting $(u,i)$ and $(w,j)$ consisting of only cross-edges of $\HH_k$, then $$N(\{(u,i),(w,j)\},\HH_k) = \frac{(m-1)^{l_k}+(-1)^{l_k}(m-1)}{m}.$$ Otherwise, $$N(\{(u,i),(w,j)\},\HH_k) = \frac{(m-1)^{l_k}-(-1)^{l_k}}{m}.$$
\end{lem}

\begin{proof}
For $k \in [n]$, since $G_k$ is a tree, $\HH_k$ has a canonical labeling. Let $r_k$ and $r_k'$ be the permutations of $[m]$ that produce a canonical labeling when $(u,i)$ is renamed to $(u,r_k(i))$ and $(w,j)$ is renamed to $(w,r_k'(j))$. Notice $r_k(i) = r_k'(j)$ if there is a path in $H_k$ connecting $(u,i)$ and $(w,j)$ consisting of only cross-edges of $\HH_k$, and $r_k(i) \not= r_k'(j)$ otherwise. If $l_1 = 1$, let $e$ be an edge with endpoints $u$ and $w$ that is distinct from the edge $uw \in E(G_1)$; otherwise, let $e = uw$. For each $k \in [n]$, let $M_k$ be the graph with $V(M_k) = V(G_k)$ and $E(M_k) = E(G_k) \cup \{e\}$. Let $M_k' = M_k \cdot e$, where we do not remove multiple edges or loops upon contraction and where the vertex obtained from contracting $e$ is $v$. Notice $M_k = C_{l_k+1}$ and $M_k' = C_{l_k}$. When there is a path in $H_k$ connecting $(u,i)$ and $(w,j)$ consisting of only cross-edges of $\HH_k$, $N(\{(u,i),(w,j)\},\HH_k)$ is the number of proper $m$-colorings of $M_k'$ that color $v$ with $r_k(i)$. Thus, $N(\{(u,i),(w,j)\},\HH_k) = P(C_{l_k},m)/m = ((m-1)^{l_k}+(-1)^{l_k}(m-1))/m$. When there is not a path in $H_k$ connecting $(u,i)$ and $(w,j)$ consisting of only cross-edges of $\HH_k$, $N(\{(u,i),(w,j)\},\HH_k)$ is the number of proper $m$-colorings of $M_k$ that color $u$ with $r_k(i)$ and $w$ with $r_k'(j)$. Thus, $N(\{(u,i),(w,j)\},\HH_k) = P(C_{l_k+1},m)/(m(m-1)) = ((m-1)^{l_k}-(-1)^{l_k})/m$.
\end{proof}

Lemma~\ref{lem: binaryThetaValues} implies the possible values of $N(\{(u,i),(w,j)\},\HH_k)$ are one apart; specifically, the possible values are: $\left((m-1)^{l_k}-(-1)^{l_k}\right)/m + (-1)^{l_k}$ and $\left((m-1)^{l_k}-(-1)^{l_k}\right)/m$.  Also, Lemmas~\ref{lem: genThetaCount} and~\ref{lem: binaryThetaValues} allow us to establish a sufficient condition for $P_{DP}(G,m)=P(G,m)$ when $G$ is a Generalized Theta graph.

\begin{lem} \label{lem: construct}
Let $G = \Theta(l_1, \ldots, l_n)$.  For any $m \geq 3$, if 

$$P(G,m) =  \left \lceil m^2 \prod_{i=1}^n \left( \left(\frac{(m-1)^{l_i}+(-1)^{l_i}(m-1)}{(m-1)^{l_i} - (-1)^{l_i}} \right)^{1/m} \left(\frac{(m-1)^{l_i}-(-1)^{l_i}}{m} \right) \right)  \right \rceil,$$

then $P_{DP}(G,m) = P(G,m)$.
\end{lem}

\begin{proof}
For some $m \geq 3$, suppose that $\mathcal{H} = (L,H)$ is a full $m$-fold cover of $G$ such that $P_{DP}(G, \mathcal{H}) = P_{DP}(G,m)$.  By Lemma~\ref{lem: genThetaCount}, we know that \\ $P_{DP}(G,\HH) = \sum_{(i,j) \in [m]^2} \prod_{k=1}^{n} N(\{(u,i),(w,j)\},\HH_k)$. By the AM-GM inequality and the fact that $P_{DP}(G,m)$ is an integer, it follows that
$$P_{DP}(G,m) \geq \left \lceil m^2  \left(\prod_{k=1}^n \prod_{(i,j) \in [m]^2}  N(\{(u,i),(w,j)\},\HH_k) \right)^{1/m^2} \right \rceil.$$
Now, suppose $l \in [n]$.  Let $H'_l$ be the spanning subgraph of $H_l$ that consists of only the cross-edges of $\HH_l$.  Since $\mathcal{H}_l$ has a canonical labeling, we know that $H'_l$ is the disjoint union of $m$ paths of length $l_l$.  Consequently, by Lemma~\ref{lem: binaryThetaValues}, $N(\{(u,i),(w,j)\},\HH_l) = ((m-1)^{l}+(-1)^{l}(m-1))/m$ for precisely $m$ ordered pairs $(i,j)$ in $[m]^2$, and $N(\{(u,i),(w,j)\},\HH_l) = ((m-1)^{l}-(-1)^{l})/m$ for all other ordered pairs in $[m]^2$.  Thus,
\begin{align*} 
&\left \lceil m^2  \left(\prod_{k=1}^n \prod_{(i,j) \in [m]^2}  N(\{(u,i),(w,j)\},\HH_k) \right)^{1/m^2} \right \rceil \\
&= \left \lceil m^2 \left(\prod_{i=1}^n \left( \left(\frac{(m-1)^{l_i}+(-1)^{l_i}(m-1)}{m} \right)^m \left(\frac{(m-1)^{l_i}-(-1)^{l_i}}{m} \right)^{m(m-1)} \right) \right)^{1/m^2} \right \rceil.
\end{align*}
So, if the hypotheses of the Lemma are satisfied, we have $P_{DP}(G,m) \geq P(G,m)$, and the result follows.
\end{proof}

The final ingredient for the proof of Theorem~\ref{thm: negative} is a result from~\cite{HK21}~\footnote{Theorem~\ref{thm: generalized} is also implied by the main results in~\cite{DY21}.}.

\begin{thm} [\cite{HK21}] \label{thm: generalized}
	Suppose $G = \Theta(l_1, \ldots, l_n)$.
	
	\textit{(i)}  If there is a $j \in \{2, \ldots, n \}$ such that $l_1$ and $l_j$ have the same parity, then there is an $N \in \N$ such that $P_{DP}(G,m) < P(G,m)$ for all $m \geq N$.
	
	\textit{(ii)}  If $l_1$ and $l_j$ have different parity for each $j \in \{2, \ldots, n \}$, then there is an $N \in \N$ such that $P_{DP}(G,m) = P(G,m)$ for all $m \geq N$. 
\end{thm}

We are now ready to prove Theorem~\ref{thm: negative}.

\begin{proof}
We will prove the result when $G= \Theta(2,3,3,3,2)$ (the proof in the other case is similar).  First, note that $P(G,3) = (6^2)(18^3)/6^4 + (6^2)(6^3)/3^4=258$.  Next, notice that if $(l_1, l_2, l_3, l_4, l_5) = (2,3,3,3,2)$, then 
$$\left \lceil 9 \prod_{i=1}^5 \left( \left(\frac{(2)^{l_i}+(-1)^{l_i}(2)}{(2)^{l_i} - (-1)^{l_i}} \right)^{1/3} \left(\frac{(2)^{l_i}-(-1)^{l_i}}{3} \right) \right)  \right \rceil = \left \lceil 9 \left(\frac{6^5}{3^2 \cdot 9^3} \right)^{1/3} \left ( \frac{3^2 \cdot 9^3}{3^5} \right )   \right \rceil = 258.$$
So, by Lemma~\ref{lem: construct}, we have that $P_{DP}(G,3)=P(G,3)$.  Also, Theorem~\ref{thm: generalized} implies there is an $N \in \N$ such that $P_{DP}(G,m) < P(G,m)$ for all $m \geq N$.  
\end{proof}

\section{Theta Graphs and the Rearrangement Inequality} \label{theta}

In this Section we show how we can use the Rearrangement Inequality along with the tools developed in Section~\ref{counter} to give an elementary proof of Theorem~\ref{thm: generalTheta3}.  We begin with a technical lemma that follows from the Rearrangement Inequality by an extremality argument.

\begin{lem}\label{lem: 3ListRearrangement}
Suppose $x_{i,j}$ is a non-negative integer for each $i \in [3]$ and $j \in [m^2]$ where $m \geq 3$. For each $i \in [3]$, suppose $n_i = x_{i,1} = \cdots = x_{i,m} \leq x_{i,m+1} = \cdots = x_{i,m(m-1)} \leq x_{i,m(m-1)+1} = \cdots = x_{i,m^2} = n_i + 1$ for some $n_i \geq 0$, and $n_1 \leq n_2,n_3$. Let $N = \lvert \{j \in [m^2] : x_{1, j} = n_1\} \rvert \in \{m, m(m - 1)\}$.
Suppose $f$ is a permutation of $[m^2]$ defined by $f(j) = m^2 + 1 - j$ and $g$ is a permutation of $[m^2]$ defined by
\[ g(j) = \begin{cases}
    m^2 + j - N & \text{if } j \in [N]\\
    j - N & \text{otherwise.}
   \end{cases}
\]
For any permutations $\sigma_1$ and $\sigma_2$ of $[m^2]$, $$\sum_{j=1}^{m^2} x_{1,j}x_{2,f(j)}x_{3,g(j)} \leq \sum_{j=1}^{m^2} x_{1,j}x_{2,\sigma_1(j)}x_{3,\sigma_2(j)}.$$
\end{lem}

\begin{proof}
For any permutations $\sigma_1$ and $\sigma_2$ of $[m^2]$, let $M(\sigma_1,\sigma_2) = \lvert\{q\in[m^2]: x_{1,q}=n_1, x_{2,\sigma_1(q)}=n_2\}\rvert + \lvert\{q\in[m^2]: x_{1,q}=n_1, x_{3,\sigma_2(q)}=n_3\}\rvert$. 

Among all permutations of $[m^2]$, $\sigma_1$ and $\sigma_2$, that minimize $\sum_{j=1}^{m^2} \left(x_{1,j}x_{2,\sigma_1(j)}x_{3,\sigma_2(j)}\right)$, choose $\gamma_1$ and $\gamma_2$ such that $M(\gamma_1,\gamma_2)$ is as small as possible.

Suppose there exist $a,b \in [m^2]$ such that $x_{1,a} = n_1$, $x_{1,b} = n_1+1$, $x_{2,\gamma_1(a)} = n_2$, and $x_{2,\gamma_1(b)} = n_2+1$. Then let $\gamma_1'$ be defined such that $\gamma_1'(a) = \gamma_1(b)$, $\gamma_1'(b) = \gamma_1(a)$, and $\gamma_1'$ equals $\gamma_1$ otherwise. We calculate
\begin{align*}
    &\sum_{j=1}^{m^2} \left(x_{1,j}x_{2,\gamma_1(j)}x_{3,\gamma_2(j)}\right) - \sum_{j=1}^{m^2} \left(x_{1,j}x_{2,\gamma_1'(j)}x_{3,\gamma_2(j)}\right)\\
    &= x_{1,a}x_{3,\gamma_2(a)}(x_{2,\gamma_1(a)} - x_{2,\gamma_1'(a)}) + x_{1,b}x_{3,\gamma_2(b)}(x_{2,\gamma_1(b)} - x_{2,\gamma_1'(b)})\\
    &= n_1 x_{3,\gamma_2(a)}(n_2 - n_2 - 1) + (n_1 + 1) x_{3,\gamma_2(b)}(n_2 + 1 - n_2)\\
    &= (n_1 + 1) x_{3,\gamma_2(b)} - n_1 x_{3,\gamma_2(a)} = n_1(x_{3,\gamma_2(b)} - x_{3,\gamma_2(a)}) + x_{3,\gamma_2(b)}\\
    &\geq n_1(n_3 - n_3 - 1) + n_3 \geq 0. 
\end{align*}

By the choice of $\gamma_1$ and $\gamma_2$, $\sum_{j=1}^{m^2} \left(x_{1,j}x_{2,\gamma_1(j)}x_{3,\gamma_2(j)}\right) - \sum_{j=1}^{m^2} \left(x_{1,j}x_{2,\gamma_1'(j)}x_{3,\gamma_2(j)}\right)=0$. However, $M(\gamma_1',\gamma_2) = M(\gamma_1,\gamma_2)-1$ which is a contradiction.  Since a similar contradiction can be reached if we assume there are $a,b \in [m^2]$ such that $x_{1,a} = n_1$, $x_{1,b} = n_1+1$, $x_{3,\gamma_2(a)} = n_3$, and $x_{3,\gamma_2(b)} = n_3+1$, we know that if $N=m$ and $x_{1,a} = n_1$, then $x_{2,\gamma_1(a)} = n_2+1$ and $x_{3,\gamma_2(a)} = n_3 + 1$.  Similarly, we know that if $N=m(m-1)$ and $x_{2,\gamma_1(a)} = n_2+1$ or $x_{3,\gamma_2(a)} = n_3 + 1$, then $x_{1,a} = n_1$.  We will now prove the desired result when $N=m$ and when $N=m(m-1)$.

Suppose $N=m$.  Without loss of generality, assume $\gamma_1([m]) = \gamma_2([m]) = [m^2] - [m(m-1)]$. Then
\begin{align*}
    & \sum_{j=1}^{m^2} x_{1,j}x_{2,\gamma_1(j)}x_{3,\gamma_2(j)}
    = \left(\sum_{j=1}^{m} n_1(n_2 + 1)(n_3 + 1)\right) + \left(\sum_{j=m+1}^{m^2} x_{1,j}x_{2,\gamma_1(j)}x_{3,\gamma_2(j)}\right)\\
    &= \left(\sum_{j=1}^{m} n_1(n_2 + 1)(n_3 + 1)\right) + (n_1 + 1)\left(\sum_{j=m+1}^{m^2} x_{2,\gamma_1(j)}x_{3,\gamma_2(j)}\right)\\
    &= mn_1(n_2 + 1)(n_3 + 1) + (n_1 + 1)\left(\sum_{j=m+1}^{m^2} x_{2,\gamma_1(j)}x_{3,\gamma_2(j)}\right)\\
    &\geq mn_1(n_2 + 1)(n_3 + 1) + (n_1 + 1)\left(\sum_{j=1}^{m(m-1)} x_{2,m(m-1)+1-j}x_{3,j}\right) \text{ (Rearrangement Inequality)}\\
    &= mn_1(n_2 + 1)(n_3 + 1) + (n_1 + 1)\sum_{j=m+1}^{m^2} x_{2,f(j)}x_{3,g(j)} = \sum_{j=1}^{m^2} x_{1,j}x_{2,f(j)}x_{3,g(j)}
\end{align*}
as desired.

Now, suppose $N=m(m-1)$.  We know that if  $x_{1,a} = n_1 + 1$, then $x_{2,\gamma_1(a)} = n_2$ and $x_{3,\gamma_2(a)} = n_3$. Without loss of generality, assume $\gamma_1([m^2] - [m(m-1)]) = \gamma_2([m^2] - [m(m-1)]) = [m]$. Then a calculation along the lines of the previous case shows that $\sum_{j=1}^{m(m-1)} x_{1,j}x_{2,\gamma_1(j)}x_{3,\gamma_2(j)} + \sum_{j=m(m-1)+1}^{m^2} (n_1 + 1)n_2n_3 \ge n_1\left(\sum_{j=1}^{m(m-1)} x_{2,f(j)}x_{3,g(j)}\right) + m(n_1 + 1)n_2n_3$, which is the desired inequality.
\end{proof}

We now introduce some terminology that will be useful in proving Theorem~\ref{thm: generalTheta3}.  Suppose that $\bm{x} = (x_1, \ldots, x_{m^2})$ satisfies $n = x_{1} = \cdots = x_{m} \leq x_{m+1} = \cdots = x_{m(m-1)} \leq x_{m(m-1)+1} = \cdots = x_{m^2} = n + 1$ for some $n \geq 0$. We say $\bm{x}$ is \emph{odd} if exactly $m$ of its coordinates are $n$, and we say $\bm{x}$ is \emph{even} if exactly $m$ of its coordinates are $n + 1$. These terms allow us to refer to the \emph{parity} of $\bm{x}$.

The following reformulation of Lemma~\ref{lem: 3ListRearrangement} will be useful in proving Theorem~\ref{thm: generalTheta3}.

\begin{lem} \label{lem: special}
Suppose $x_{i,j}$ is a non-negative integer for each $i \in [3]$ and $j \in [m^2]$ where $m \geq 3$. For each $i \in [3]$, suppose $n_i = x_{i,1} = \cdots = x_{i,m} \leq x_{i,m+1} = \cdots = x_{i,m(m-1)} \leq x_{i,m(m-1)+1} = \cdots = x_{i,m^2} = n_i + 1$ for some $n_i \geq 0$, and $n_1 \leq n_2,n_3$. Let $\bm{x}_i = (x_{i, 1}, \ldots, x_{i, m^2})$ for each $i \in [3]$. For each $i \in [3]$, if $\bm{x}_i$ is odd, let $s_i = n_i$ and $o_i = n_i + 1$, and if $\bm{x}_i$ is even, let $s_i = n_i + 1$ and $o_i = n_i$. Let $h_1$ be a permutation of $[m^2]$ such that $x_{1, h_1(j)} = s_1$ whenever $j \in [m]$ and $x_{1, h_1(j)} = o_1$ otherwise. We now define two more permutations $h_2$ and $h_3$ of $[m^2]$.

(i) If the parity of $\bm{x}_1$ is different from both $\bm{x}_2$ and $\bm{x}_3$, then for each $i \in \{2, 3\}$, let $h_i$ be a permutation of $[m^2]$ such that $x_{i, h_i(j)} = s_i$ whenever $j \in [m]$ and $x_{i, h_i(j)} = o_i$ otherwise.

(ii) If the parity of $\bm{x}_1$ is different from $\bm{x}_2$ and the same as $\bm{x}_3$, then let $h_2$ be a permutation of $[m^2]$ such that $x_{2, h_2(j)} = s_2$ whenever $j \in [m]$ and $x_{2, h_2(j)} = o_2$ otherwise, and let $h_3$ be a permutation of $[m^2]$ such that $x_{3, h_3(j)} = s_3$ whenever $j \in [2m] - [m]$ and $x_{3, h_3(j)} = o_3$ otherwise.

(iii) If $\bm{x}_1$, $\bm{x}_2$, and $\bm{x}_3$ have the same parity, then let $h_2$ be a permutation of $[m^2]$ such that $x_{2, h_2(j)} = s_2$ whenever $j \in [2m] - [m]$ and $x_{2, h_2(j)} = o_2$ otherwise, and let $h_3$ be a permutation of $[m^2]$ such that $x_{3, h_3(j)} = s_3$ whenever $j \in [3m] - [2m]$ and $x_{3, h_3(j)} = o_3$ otherwise.

Then, using the notation of Lemma~\ref{lem: 3ListRearrangement}, $$\sum_{j = 1}^{m^2} x_{1, h_1(j)} x_{2, h_2(j)} x_{3, h_3(j)} = \sum_{j=1}^{m^2} x_{1,j}x_{2,f(j)}x_{3,g(j)}.$$
In particular, for any permutations $\sigma_1, \sigma_2$ of $[m^2]$, $$ \sum_{j = 1}^{m^2} x_{1, h_1(j)} x_{2, h_2(j)} x_{3, h_3(j)} \leq \sum_{j = 1}^{m^2} x_{1, j} x_{2, \sigma_1(j)} x_{3, \sigma_2(j)}. $$
\end{lem}

\begin{proof}
Define $f$, $g$, and $N$ as in Lemma~\ref{lem: 3ListRearrangement}. We consider six cases corresponding to the parities of $\bm{x}_1$, $\bm{x}_2$, and $\bm{x}_3$ in (i), (ii), and (iii).

For (i), for suppose $\bm{x}_1$ is even, $\bm{x}_2$ is odd, and $\bm{x}_3$ is odd. Then we have the following:

\noindent\begin{minipage}{.25\linewidth}
\begin{align*}
    x_{1, h_1(j)} =
    \begin{cases}
    n_1 + 1\\
    n_1
    \end{cases}
\end{align*}
\end{minipage}
\begin{minipage}{.25\linewidth}
\begin{align*}
    , \; x_{2, h_2(j)} =
    \begin{cases}
    n_2 \\
    n_2 + 1 
    \end{cases}
\end{align*}
\end{minipage}
\begin{minipage}{.5\linewidth}
\begin{align*}
    , \; x_{3, h_3(j)} =
    \begin{cases}
    n_3 & \;\; \text{if } j \in [m] \\
    n_3 + 1 & \;\; \text{if } j \in [m^2] - [m];
    \end{cases}
\end{align*}
\end{minipage}

\noindent\begin{minipage}{.25\linewidth}
\begin{align*}
    x_{1, j} =
    \begin{cases}
    n_1  \\
    n_1 + 1
    \end{cases}
\end{align*}
\end{minipage}
\begin{minipage}{.25\linewidth}
\begin{align*}
    , \; x_{2, f(j)} =
    \begin{cases}
    n_2 + 1 \\
    n_2 
    \end{cases}
\end{align*}
\end{minipage}
\begin{minipage}{.5\linewidth}
\begin{align*}
    , \; x_{3, g(j)} =
    \begin{cases}
    n_3 + 1 & \;\; \text{if } j \in [m(m - 1)] \\
    n_3 & \;\; \text{if } j \in [m^2] - [m(m - 1)].
    \end{cases}
\end{align*}
\end{minipage}
Therefore, we obtain
\begin{align*}
    \sum_{j = 1}^{m^2} x_{1, h_1(j)} x_{2, h_2(j)} x_{3, h_3(j)}
    &= \sum_{j = 1}^{m} (n_1 + 1)n_2n_3 + \sum_{j = m + 1}^{m^2} n_1(n_2 + 1)(n_3 + 1) \\
    &= m(n_1 + 1)n_2n_3 + m(m - 1)n_1(n_2 + 1)(n_3 + 1) \\
    &= \sum_{j = 1}^{m(m - 1)} n_1(n_2 + 1)(n_3 + 1) + \sum_{j = m(m - 1) + 1}^{m^2} (n_1 + 1)n_2n_3 \\
    &= \sum_{j = 1}^{m^2} x_{1, j} x_{2, f(j)} x_{3, g(j)}.
\end{align*}

The other possibility for (i) is that $\bm{x}_1$ is odd, $\bm{x}_2$ is even, and $\bm{x}_3$ is even. The details of the proof are similar to the previous case and are given in Appendix~\ref{A}.

For (ii), first we suppose $\bm{x}_1$ is even, $\bm{x}_2$ is odd, and $\bm{x}_3$ is even. Then we have:

\noindent\begin{minipage}{.25\linewidth}
\begin{align*}
    x_{1, h_1(j)} =
    \begin{cases}
    n_1 + 1 \\
    n_1
    \end{cases}
\end{align*}
\end{minipage}
\begin{minipage}{.5\linewidth}
\begin{align*}
    , \; x_{2, h_2(j)} =
    \begin{cases}
    n_2 & \;\; \text{if } j \in [m] \\
    n_2 + 1 & \;\; \text{if } j \in [m^2] - [m];
    \end{cases}
\end{align*}
\end{minipage}
\begin{align*}
    x_{3, h_3(j)} =
    \begin{cases}
    n_3 + 1 & \text{if } j \in [2m] - [m] \\
    n_3 & \text{if } j \in [m] \cup \left( [m^2] - [2m] \right);
    \end{cases}
\end{align*}

\noindent\begin{minipage}{.25\linewidth}
\begin{align*}
    x_{1, j} =
    \begin{cases}
    n_1 \\
    n_1 + 1
    \end{cases}
\end{align*}
\end{minipage}
\begin{minipage}{.5\linewidth}
\begin{align*}
    , \;\; x_{2, f(j)} =
    \begin{cases}
    n_2 + 1 & \; \text{if } j \in [m(m - 1)] \\
    n_2 & \; \text{if } j \in [m^2] - [m(m - 1)].
    \end{cases}
\end{align*}
\end{minipage}

We also have $N = m(m - 1)$, and so
\begin{align*}
    x_{3, g(j)} =
    \begin{cases}
    n_3 & \text{if } j \in [m(m - 2)] \cup \left( [m^2] - [m(m - 1)] \right) \\
    n_3 + 1 & \text{if } j \in [m(m - 1)] - [m(m - 2)].
    \end{cases}
\end{align*}
Therefore, we obtain
\begin{align*}
    &\sum_{j = 1}^{m^2} x_{1, h_1(j)} x_{2, h_2(j)} x_{3, h_3(j)} \\
    &= \sum_{j = 1}^m (n_1 + 1)n_2n_3 + \sum_{j = m + 1}^{2m} n_1(n_2 + 1)(n_3 + 1) + \sum_{j = 2m + 1}^{m^2} n_1(n_2 + 1)n_3 \\
    &= m(n_1 + 1)n_2n_3 + mn_1(n_2 + 1)(n_3 + 1) + m(m - 2)n_1(n_2 + 1)n_3 \\
    &= \sum_{j = 1}^{m(m - 2)} n_1(n_2 + 1)n_3 + \sum_{j = m(m - 2) + 1}^{m(m - 1)} n_1(n_2 + 1)(n_3 + 1) + \sum_{j = m(m - 1) + 1}^{m^2} (n_1 + 1)n_2n_3 \\
    &= \sum_{j = 1}^{m^2} x_{1, j} x_{2, f(j)} x_{3, g(j)}.
\end{align*}

The other possibility for (ii) is that $\bm{x}_1$ is odd, $\bm{x}_2$ is even, and $\bm{x}_3$ is odd.  The details of the proof are similar to the previous case and are given in Appendix~\ref{A}.

Finally, turning our attention to~(iii), we suppose $\bm{x}_1$, $\bm{x}_2$, and $\bm{x}_3$ are all even. Then we have the following:
\begin{align*}
    x_{1, h_1(j)} =
    \begin{cases}
    n_1 + 1 & \text{if } j \in [m] \\
    n_1 & \text{if } j \in [m^2] - [m];
    \end{cases}
\end{align*}
\begin{align*}
    x_{2, h_2(j)} =
    \begin{cases}
    n_2 + 1 & \text{if } j \in [2m] - [m] \\
    n_2 & \text{if } j \in [m] \cup \left( [m^2] - [2m] \right);
    \end{cases}
\end{align*}
\begin{align*}
    x_{3, h_3(j)} =
    \begin{cases}
    n_3 + 1 & \text{if } j \in [3m] - [2m] \\
    n_3 & \text{if } j \in [2m] \cup \left( [m^2] - [3m] \right);
    \end{cases}
\end{align*}
\begin{align*}
    x_{1, j} =
    \begin{cases}
    n_1 & \text{if } j \in [m(m - 1)] \\
    n_1 + 1 & \text{if } j \in [m^2] - [m(m - 1)];
    \end{cases}
\end{align*}
\begin{align*}
    x_{2, f(j)} =
    \begin{cases}
    n_2 + 1 & \text{if } j \in [m] \\
    n_2 & \text{if } j \in [m^2] - [m].
    \end{cases}
\end{align*}
We also have $N = m(m - 1)$, and so
\begin{align*}
    x_{3, g(j)} =
    \begin{cases}
    n_3 & \text{if } j \in [m(m - 2)] \cup \left( [m^2] - [m(m - 1)] \right) \\
    n_3 + 1 & \text{if } j \in [m(m - 1)] - [m(m - 2)].
    \end{cases}
\end{align*}
Therefore, we obtain
\begin{align*}
    &\sum_{j = 1}^{m^2} x_{1, h_1(j)} x_{2, h_2(j)} x_{3, h_3(j)} \\
    &= \sum_{j = 1}^m (n_1 + 1)n_2n_3 + \sum_{j = m + 1}^{2m} \mkern-10mu n_1(n_2 + 1)n_3 + \sum_{j = 2m + 1}^{3m} \mkern-16mu n_1n_2(n_3 + 1) + \sum_{j = 3m + 1}^{m^2} \mkern-16mu n_1n_2n_3 \\
    &= m(n_1 + 1)n_2n_3 + mn_1(n_2 + 1)n_3 + mn_1n_2(n_3 + 1) + m(m - 3)n_1n_2n_3 \\
    &= \sum_{j = 1}^m n_1(n_2 + 1)n_3 + \sum_{j = m + 1}^{m(m - 2)} \mkern-12mu n_1n_2n_3 + \sum_{j = m(m - 2) + 1}^{m(m - 1)} \mkern-20mu n_1n_2(n_3 + 1) + \sum_{j = m(m - 1) + 1}^{m^2} \mkern-20mu (n_1 + 1)n_2n_3 \\
    &= \sum_{j = 1}^{m^2} x_{1, j} x_{2, f(j)} x_{3, g(j)}.
\end{align*}

The other possibility for (iii) is that $\bm{x}_1$, $\bm{x}_2$, and $\bm{x}_3$ are all odd. The details of the proof are similar to the previous case and are given in Appendix~\ref{A}.
\end{proof}

We are now ready to show that each formula in Theorem~\ref{thm: generalTheta3} is a lower bound on the DP color function of the appropriate Theta graph.

\begin{lem} \label{lem: theta-dp-lower-bound} 
Let $G = \Theta(l_1,l_2,l_3)$ and $\HH = (L,H)$ be a full $m$-fold cover of $G$ where $m \geq 3$.

(i) If the parity of $l_1$ is different from both $l_2$ and $l_3$, then $P_{DP}(G, \mathcal{H}) \geq P(G, m)$.

(ii) If the parity of $l_1$ is different from $l_2$ and the same as $l_3$, then

\noindent $\displaystyle P_{DP}(G, \mathcal{H}) \geq \frac{1}{m} \left[ (m - 1)^{l_1 + l_2 + l_3} + (m - 1)^{l_1} - (m - 1)^{l_2} - (m - 1)^{l_3 + 1} + (-1)^{l_2 + 1}(m - 2) \right]$.

(iii) If $l_1$, $l_2$, and $l_3$ all have the same parity, then

\noindent $\displaystyle P_{DP}(G, \mathcal{H}) \geq \frac{1}{m} \left[ (m - 1)^{l_1 + l_2 + l_3} - (m - 1)^{l_1} - (m - 1)^{l_2} - (m - 1)^{l_3} + 2(-1)^{l_1 + l_2 + l_3} \right]$.
\end{lem}

\begin{proof}

We begin by using Lemmas~\ref{lem: genThetaCount} and~\ref{lem: binaryThetaValues} to find a formula for $P_{DP}(G, \HH)$ to which we can apply Lemma~\ref{lem: special}. For each $k \in [3]$ and $(i, j) \in [m]^2$, let $s_{k, (i, j)} = N(\{(u, i), (w, j)\}, \mathcal{H}_{k})$. Consider some $k \in [3]$ and $(i, j) \in [m]^2$. By Lemma~\ref{lem: binaryThetaValues}, we know that
\begin{align} \label{s_eq1}
    s_{k, (i, j)} = \frac{(m - 1)^{l_k} + (-1)^{l_k}(m - 1)}{m} = \frac{(m - 1)^{l_k} - (-1)^{l_k}}{m} + (-1)^{l_k}
\end{align}
if there is a path in $H_k$ from $(u, i)$ to $(w, j)$ consisting only of cross-edges of $\HH_k$, and
\begin{align} \label{s_eq2}
    s_{k, (i, j)} = \frac{(m - 1)^{l_k} - (-1)^{l_k}}{m}
\end{align}
otherwise. In particular, notice that Equation~(\ref{s_eq1}) holds for exactly $m$ choices of $(i, j) \in [m]^2$; whereas, Equation~(\ref{s_eq2}) holds for the remaining $m(m - 1)$ choices of $(i, j) \in [m]^2$.

For each $k \in [3]$, let $n_k = \min_{(i, j) \in [m]^2} s_{k, (i, j)}$. Notice that, for each $(i, j) \in [m]^2$, either $s_{k, (i, j)} = n_k$ or $s_{k, (i, j)} = n_k + 1$. Moreover, we have $s_{k, (i, j)} = n_k$ for either $m$ or $m(m - 1)$ choices of $(i, j) \in [m]^2$, while $s_{k, (i, j)} = n_k + 1$ for the remaining $m(m - 1)$ or $m$ choices, respectively, of $(i, j) \in [m]^2$.

Now, let $\beta : [m]^2 \rightarrow [m^2]$ be the function defined by $\beta(i, j) = m(i - 1) + j$ for each $(i, j) \in [m]^2$. Notice that $\beta$ is bijective and hence has an inverse $\beta^{-1} : [m^2] \rightarrow [m]^2$. By Lemma~\ref{lem: genThetaCount},
\begin{align*}
    P_{DP}(G, \mathcal{H}) = \sum_{(i, j) \in [m]^2} s_{1, (i, j)} s_{2, (i, j)} s_{3, (i, j)} = \sum_{\ell = 1}^{m^2} s_{1, \beta^{-1}(\ell)} s_{2, \beta^{-1}(\ell)} s_{3, \beta^{-1}(\ell)}.
\end{align*}
If we let $s_{k, \ell} = s_{k, \beta^{-1}(\ell)}$ for each $k \in [3]$ and $\ell \in [m^2]$, then we can alternatively write
\begin{align*}
    P_{DP}(G, \mathcal{H}) = \sum_{\ell = 1}^{m^2} s_{1, \ell} s_{2, \ell} s_{3, \ell}.
\end{align*}

For each $k \in [3]$, let $\rho_k : [m^2] \rightarrow [m^2]$ be any permutation of $[m^2]$ such that for each $\ell_1, \ell_2 \in [m^2]$ with $\ell_1 < \ell_2$, we have $s_{k, \rho_k(\ell_1)} \leq s_{k, \rho_k(\ell_2)}$. Furthermore, let $x_{k, \ell} = s_{k, \rho_k(\ell)}$ for each $\ell \in [m^2]$. By definition, for each $\ell_1, \ell_2 \in [m^2]$ with $\ell_1 < \ell_2$, we have $x_{k, \ell_1} \leq x_{k, \ell_2}$. Moreover, notice that $x_{k, 1} = \cdots = x_{k, m} = n_k$ and $x_{k, m(m - 1) + 1} = \cdots = x_{k, m^2} = n_k + 1$. Finally, notice that $\rho_k$ is bijective, and so it has an inverse $\rho_k^{-1} : [m^2] \rightarrow [m^2]$. Using this notation, we may write
\begin{align*}
    P_{DP}(G, \mathcal{H}) &= \sum_{\ell = 1}^{m^2} s_{1, \rho_1(\ell)} s_{2, \rho_2(\rho_2^{-1}(\rho_1(\ell)))} s_{3, \rho_3(\rho_3^{-1}(\rho_1(\ell)))}
    = \sum_{\ell = 1}^{m^2} x_{1, \ell} x_{2, \rho_2^{-1}(\rho_1(\ell))} x_{3, \rho_3^{-1}(\rho_1(\ell))}.
\end{align*}
If we let $\sigma_1, \sigma_2 : [m^2] \rightarrow [m^2]$ be the permutations of $[m^2]$ defined by $\sigma_1(\ell) = \rho_2^{-1}(\rho_1(\ell))$ and $\sigma_2(\ell) = \rho_3^{-1}(\rho_1(\ell))$ for each $\ell \in [m^2]$, then we can alternatively write
\begin{align*}
    P_{DP}(G, \mathcal{H}) = \sum_{\ell = 1}^{m^2} x_{1, \ell} x_{2, \sigma_1(\ell)} x_{3, \sigma_2(\ell)}.
\end{align*}

Define $\bm{x}_k$, $h_k$, $s_k$, and $o_k$ for each $k \in [3]$ as in Lemma~\ref{lem: special}. By Lemma~\ref{lem: special},
\begin{align} \label{dp-lower-bound}
    P_{DP}(G, \mathcal{H}) \geq \sum_{\ell = 1}^{m^2} x_{1, h_1(\ell)} x_{2, h_2(\ell)} x_{3, h_3(\ell)}.
\end{align}

Now, we claim for each $k \in [3]$,
\begin{align} \label{sk-formula}
    s_k = \frac{(m - 1)^{l_k} - (-1)^{l_k}}{m} + (-1)^{l_k} = \frac{(m - 1)^{l_k} + (-1)^{l_k}(m - 1)}{m}
\end{align}
and
\begin{align} \label{ok-formula}
    o_k = \frac{(m - 1)^{l_k} + (-1)^{l_k + 1}}{m}.
\end{align}  
To see why these formulas hold, consider the case where $l_k$ is even and the case where $l_k$ is odd separately.  Suppose that $l_k$ is even. Then we have
\begin{align*}
    s_{k, (i, j)} = \frac{(m - 1)^{l_k} - (-1)^{l_k}}{m} + 1
\end{align*}
for exactly $m$ choices of $(i, j) \in [m]^2$, and
\begin{align*}
    s_{k, (i, j)} = \frac{(m - 1)^{l_k} - (-1)^{l_k}}{m}
\end{align*}
for the remaining $m(m - 1)$ choices of $(i, j) \in [m]^2$. Therefore, we have
\begin{align*}
    n_k = \frac{(m - 1)^{l_k} - (-1)^{l_k}}{m}
\end{align*}
and
\begin{align*}
    x_{k, \ell} =
    \begin{cases}
    n_k & \text{if } \ell \in [m(m - 1)] \\
    n_k + 1 & \text{if } \ell \in [m^2] - [m(m - 1)].
    \end{cases}
\end{align*}
It follows by definition that $\bm{x}_k$ is even, so that
\begin{align*}
    s_k = n_k + 1 = \frac{(m - 1)^{l_k} - (-1)^{l_k}}{m} + (-1)^{l_k}\;\;
\text{and }
    o_k = n_k = \frac{(m - 1)^{l_k} - (-1)^{l_k}}{m}.
\end{align*}

The proof for the case $l_k$ is odd follows similarly. The details are given in Appendix~\ref{B}.

We are now ready to prove Statements~(i), (ii), and (iii). For Statement~(i), suppose that the parity of $l_1$ is different from that of both $l_2$ and $l_3$. Then, for each $k \in [3]$, we have
\begin{align*}
    x_{k, h_k(\ell)} =
    \begin{cases}
    s_k & \text{if } \ell \in [m] \\
    o_k & \text{if } \ell \in [m^2] - [m].
    \end{cases}
\end{align*}
So,
\begin{align*}
    P_{DP}(G, \mathcal{H}) &\geq \sum_{\ell = 1}^{m^2} x_{1, h_1(\ell)} x_{2, h_2(\ell)} x_{3, h_3(\ell)} = \sum_{\ell = 1}^m s_1s_2s_3 + \sum_{\ell = m + 1}^{m^2} o_1o_2o_3 \\
    &= ms_1s_2s_3 + m(m - 1)o_1o_2o_3.
\end{align*}

For Statement~(ii), suppose that the parity of $l_1$ is different from that of $l_2$ and the same as that of $l_3$. Then we have
\begin{align*}
    x_{k, h_k(\ell)} =
    \begin{cases}
    s_k & \text{if } \ell \in [m] \\
    o_k & \text{if } \ell \in [m^2] - [m]
    \end{cases}
\end{align*}
for $k \in [2]$ and
\begin{align*}
    x_{3, h_3(\ell)} =
    \begin{cases}
    s_k & \text{if } \ell \in [2m] - [m] \\
    o_k & \text{if } \ell \in [m] \cup \left( [m^2] - [2m] \right).
    \end{cases}
\end{align*}
So,
\begin{align*}
    P_{DP}(G, \mathcal{H}) &\geq \sum_{\ell = 1}^{m^2} x_{1, h_1(\ell)} x_{2, h_2(\ell)} x_{3, h_3(\ell)} = \sum_{\ell = 1}^m s_1 s_2 o_3 + \sum_{\ell = m + 1}^{2m} o_1 o_2 s_3 + \sum_{\ell = 2m + 1}^{m^2} o_1 o_2 o_3 \\
    &=ms_1s_2o_3 + mo_1o_2s_3 + m(m - 2)o_1o_2o_3.
\end{align*}

For Statement~(iii), suppose that $l_1$, $l_2$, and $l_3$ all have the same parity. Then, for each $k \in [3]$, we have
\begin{align*}
    x_{k, h_k(\ell)} =
    \begin{cases}
    s_k & \text{if } \ell \in [km] - [(k - 1)m] \\
    o_k & \text{if } \ell \in [(k - 1)m] \cup \left( [m^2] - [km] \right).
    \end{cases}
\end{align*}
So,
\begin{align*}
    P_{DP}(G, \mathcal{H}) &\geq \sum_{\ell = 1}^{m^2} x_{1, h_1(\ell)} x_{2, h_2(\ell)} x_{3, h_3(\ell)}\\
    &= \sum_{\ell = 1}^m s_1 o_2 o_3 + \sum_{\ell = m + 1}^{2m} o_1 s_2 o_3 + \sum_{\ell = 2m + 1}^{3m} o_1 o_2 s_3 + \sum_{\ell = 3m + 1}^{m^2} o_1 o_2 o_3 \\
    &= ms_1o_2o_3 + mo_1s_2o_3 + mo_1o_2s_3 + m(m - 3)o_1o_2o_3.
\end{align*}

Finally, substituting the formulas~\ref{sk-formula} and~\ref{ok-formula} for $s_k$ and $o_k$ into the three lower bounds that we just obtained yields the appropriate formula after some algebraic simplification.
\end{proof} 

Having established the appropriate lower bounds, we are now ready to complete the proof of Theorem~\ref{thm: generalTheta3}.

\begin{proof}
Since $G$ contains a cycle, $P_{DP}(G, 1) = P_{DP}(G, 2) = 0$. So, the result holds when $m = 1, 2$. Therefore, throughout this proof we suppose that $m \geq 3$. Notice that if the parity of $l_1$ is different from both $l_2$ and $l_3$, then we know from Lemma~\ref{lem: theta-dp-lower-bound} that $P_{DP}(G, m) \geq P(G, m)$. Since we also know that $P_{DP}(G, m) \leq P(G, m)$, Statement~(i) follows.

For the remaining two statements, we will construct a full $m$-fold cover $\HH = (L,H)$ of $G$ with an appropriate number of $\HH$-colorings. For each $i \in [3]$, suppose the vertices of $R_i$ written in order are: $u,v_{i,1},\ldots,v_{i,l_i-1},w$. In the case $l_1 = 1$, $R_1$ has no internal vertices. Let $G' = G - \{v_{2,l_2-1}w,v_{3,l_3-1}w\}$. Suppose $\HH' = (L,H')$ is an $m$-fold cover of $G'$ with a canonical labeling. Begin constructing edges of $H$ by including all the edges in $E(H')$. The remaining edges between $L(v_{2,l_2-1})$ and $L(w)$, and between $L(v_{3,l_3-1})$ and $L(w)$, will be specified below.

Let $\sigma : [m] \rightarrow [m]$ be the permutation of $[m]$ given by $\sigma(j) = (j \text{ (mod } m)) + 1$.  Suppose $l_3$ has the same parity as $l_1$ and $l_2$ has different parity than $l_1$. Complete the construction of $H$ by including $\{(v_{2,l_2 - 1},j)(w,j) : j \in [m]\} \cup \{(v_{3,l_3 - 1},j)(w,\sigma(j)) : j \in [m]\}$ in $E(H)$.
By Lemma~\ref{lem: binaryThetaValues}, for each $k \in [2]$ we have $N(\{(u,i),(w,j)\},\HH_k) = ((m-1)^{l_k} + (-1)^{l_k}(m-1))/m$ when $i = j$ and $N(\{(u,i),(w,j)\},\HH_k) = ((m-1)^{l_k} - (-1)^{l_k})/m$ otherwise.
Again by Lemma~\ref{lem: binaryThetaValues}, we have $N(\{(u,i),(w,j)\},\HH_3) = ((m-1)^{l_3} + (-1)^{l_3}(m-1))/m$ when $j = \sigma(i)$ and $N(\{(u,i),(w,j)\},\HH_3) = ((m-1)^{l_3} - (-1)^{l_3})/m$ otherwise. Thus, by Lemma~\ref{lem: genThetaCount},
\begin{align*}
    P_{DP}(G,\HH) &= \sum_{(i,j) \in [m]^2} \prod_{k=1}^{3} N(\{(u,i),(w,j)\},\HH_k)\\
    &= m\left(\frac{(m-1)^{l_1} + (-1)^{l_1}(m-1)}{m}\right)\left(\frac{(m-1)^{l_2} + (-1)^{l_2}(m-1)}{m}\right)\left(\frac{(m-1)^{l_3} - (-1)^{l_3}}{m}\right)\\
    &+ m\left(\frac{(m-1)^{l_1} - (-1)^{l_1}}{m}\right)\left(\frac{(m-1)^{l_2} - (-1)^{l_2}}{m}\right)\left(\frac{(m-1)^{l_3} + (-1)^{l_3}(m-1)}{m}\right)\\
    &+ (m^2 - 2m)\left(\frac{(m-1)^{l_1} - (-1)^{l_1}}{m}\right)\left(\frac{(m-1)^{l_2} - (-1)^{l_2}}{m}\right)\left(\frac{(m-1)^{l_3} - (-1)^{l_3}}{m}\right)\\
    &= \frac{1}{m}\left((m-1)^{l_1+l_2+l_3} + (m-1)^{l_1} - (m-1)^{l_2} - (m-1)^{l_3+1} - (-1)^{l_2}(m-2)\right).
\end{align*}

Suppose $l_2$ and $l_3$ have the same parity as $l_1$. Complete the construction of $H$ by including $\{(v_{2,l_2 - 1},j)(w,\sigma(j)) : j \in [m]\} \cup \{(v_{3,l_3 - 1},j)(w,\sigma^2(j)) : j \in [m]\}$ in $E(H)$ \footnote{Throughout this document, whenever $\sigma$ is a permutation of $[m]$ and $k \in \mathbb{N}$, we write $\sigma^{k}$ for $\sigma \circ \cdots \circ \sigma$, where $\sigma$ appears $k$ times. Moreover, we write $\sigma^0$ for the identity map on $[m]$.}.
By Lemma~\ref{lem: binaryThetaValues}, for each $k \in [3]$ we have $N(\{(u,i),(w,j)\},\HH_k) = ((m-1)^{l_k} + (-1)^{l_k}(m-1))/m$ when $i = \sigma^{k - 1}(j)$ and $N(\{(u,i),(w,j)\},\HH_k) = ((m-1)^{l_k} - (-1)^{l_k})/m$ otherwise.
Thus, by Lemma~\ref{lem: genThetaCount},
\begin{align*}
    P_{DP}(G,\HH) &= \sum_{(i,j) \in [m]^2} \prod_{k=1}^{3} N(\{(u,i),(w,j)\},\HH_k)\\
    &= m\left(\frac{(m-1)^{l_1} + (-1)^{l_1}(m-1)}{m}\right)\left(\frac{(m-1)^{l_2} - (-1)^{l_2}}{m}\right)\left(\frac{(m-1)^{l_3} - (-1)^{l_3}}{m}\right)\\
    &+ m\left(\frac{(m-1)^{l_1} - (-1)^{l_1}}{m}\right)\left(\frac{(m-1)^{l_2} + (-1)^{l_2}(m-1)}{m}\right)\left(\frac{(m-1)^{l_3} - (-1)^{l_3}}{m}\right)\\
    &+ m\left(\frac{(m-1)^{l_1} - (-1)^{l_1}}{m}\right)\left(\frac{(m-1)^{l_2} - (-1)^{l_2}}{m}\right)\left(\frac{(m-1)^{l_3} + (-1)^{l_3}(m-1)}{m}\right)\\
    &+ (m^2 - 3m)\left(\frac{(m-1)^{l_1} - (-1)^{l_1}}{m}\right)\left(\frac{(m-1)^{l_2} - (-1)^{l_2}}{m}\right)\left(\frac{(m-1)^{l_3} - (-1)^{l_3}}{m}\right)\\
    &= \frac{1}{m}\left((m-1)^{l_1+l_2+l_3} - (m-1)^{l_1} - (m-1)^{l_2} - (m-1)^{l_3} + 2(-1)^{l_1+l_2+l_3}\right).
\end{align*}
\end{proof}

\section{The Dual DP Color Function of Generalized Theta Graphs} \label{dual}

In this Section we show how the ideas we have developed thus far can be used to completely determine the dual DP color function of all Generalized Theta graphs.  In particular, we prove the following theorem.

\begin{thm} \label{thm: dual}
Let $G = \Theta(l_1, \ldots, l_n)$, where $n \geq 2$, $l_1 = \min_{i \in [n]} l_i \geq 1$, and $l_i \geq 2$ for each $i \in [n] - \{1\}$. If $l_1, \ldots, l_n$ do not all have the same parity, let $t = \max\{i \in [n] : (l_1 - l_i) \textup{ mod } 2 = 1\}$; otherwise, let $t = 1$. Let $m \geq 2$. For each $i \in [n]$, let
\begin{align*}
    n_i =
    \begin{cases}
    \displaystyle \frac{(m - 1)^{l_i} - (-1)^{l_i}}{m} & \text{if } l_i \text{ is even} \\[1em]
    \displaystyle \frac{(m - 1)^{l_i} + (-1)^{l_i}(m - 1)}{m} & \text{if } l_i \text{ is odd}.
    \end{cases}
\end{align*}
If $l_i$ is even, let $s_i = n_i + 1$ and $o_i = n_i$. If $l_i$ is odd, let $s_i = n_i$ and $o_i = n_i + 1$. Let $S = \prod_{i \in [n]} s_i$ and $O = \prod_{i \in [n]} o_i$. Then
\begin{align*}
    P_{DP}^*(G, m) = \frac{mS}{\prod_{i = 2}^t s_i} \prod_{i = 2}^t o_i + m(m - 2)O + \frac{mO}{\prod_{i = 2}^t o_i} \prod_{i = 2}^t s_i.
\end{align*}
\end{thm}  

It is worth mentioning that Theorem~\ref{thm: dual} implies that when $t=1$, $P_{DP}^*(G, m) = P(G,m)$ (There is an analogous result for the DP color function; specifically, see Statement~(i) of Theorem~\ref{thm: generalized}.). The proof of Theorem~\ref{thm: dual} uses a lemma that follows immediately from the second part of the Rearrangement Inequality.

\begin{lem} \label{lem: rearrangementUpperBound}
Let $n, m \geq 2$. Suppose $x_{i, j}$ is a non-negative integer for each $i \in [n]$ and $j \in [m^2]$. For each $i \in [n]$, suppose $n_i = x_{i, 1} = \cdots = x_{i, m} \leq x_{i, m + 1} = \cdots = x_{i, m(m - 1)} \leq x_{i, m(m - 1) + 1} = \cdots = x_{i, m^2} = n_i + 1$ for some $n_i \geq 0$. Then for any permutations $\sigma_1, \ldots, \sigma_n$ of $[m^2]$,
\begin{align*}
    \sum_{j = 1}^{m^2} \prod_{i = 1}^n x_{i, \sigma_i(j)} \leq \sum_{j = 1}^{m^2} \prod_{i = 1}^n x_{i, j}.
\end{align*}
\end{lem}

We are now ready to show that each formula in Theorem~\ref{thm: dual} is an upper bound on the dual DP color function of the appropriate Generalized Theta graph.

\begin{lem} \label{lem: dual-dp-upper-bound}
Let $G = \Theta(l_1, \ldots, l_n)$ and $\mathcal{H}$ be a full $m$-fold cover of $G$. Then, using the same notation as Theorem~\ref{thm: dual},
\begin{align*}
    P_{DP}(G, \mathcal{H}) \leq \frac{mS}{\prod_{i = 2}^t s_i} \prod_{i = 2}^t o_i + m(m - 2)O + \frac{mO}{\prod_{i = 2}^t o_i} \prod_{i = 2}^t s_i.
\end{align*}
\end{lem}

\begin{proof}
We begin by using Lemmas~\ref{lem: genThetaCount} and~\ref{lem: binaryThetaValues} to find a formula for $P_{DP}(G, \HH)$ to which we can apply Lemma~\ref{lem: rearrangementUpperBound}.  For each $k \in [n]$ and $(i, j) \in [m]^2$ let $s_{k, (i, j)} = N(\{(u, i), (w, j)\}, \mathcal{H}_k)$. Consider some $k \in [n]$ and $(i, j) \in [m]^2$. By Lemma~\ref{lem: binaryThetaValues}, we know that
\begin{align} \label{s_eq4}
    s_{k, (i, j)} = \frac{(m - 1)^{l_k} + (-1)^{l_k}(m - 1)}{m} = \frac{(m - 1)^{l_k} - (-1)^{l_k}}{m} + (-1)^{l_k}
\end{align}
if there is a path in $H_k$ from $(u, i)$ to $(w, j)$ consisting only of cross-edges of $\HH_k$, and
\begin{align} \label{s_eq5}
    s_{k, (i, j)} = \frac{(m - 1)^{l_k} - (-1)^{l_k}}{m}
\end{align}
otherwise. In particular, notice that Equation (\ref{s_eq4}) holds for exactly $m$ choices of $(i, j) \in [m]^2$, whereas Equation (\ref{s_eq5}) holds for the remaining $m(m - 1)$ choices of $(i, j) \in [m]^2$.

For each $k \in [n]$, let $n_k = \min_{(i, j) \in [m]^2} s_{k, (i, j)}$. Notice that for each $(i, j) \in [m]^2$, either $s_{k, (i, j)} = n_k$ or $s_{k, (i, j)} = n_k + 1$. Moreover, we have $s_{k, (i, j)} = n_k$ for either $m$ or $m(m - 1)$ choices of $(i, j) \in [m]^2$, while $s_{k, (i, j)} = n_k + 1$ for the remaining $m(m - 1)$ or $m$ choices, respectively, of $(i, j) \in [m]^2$.

Now, let $\beta : [m]^2 \rightarrow [m^2]$ be the function defined by $\beta(i, j) = m(i - 1) + j$ for each $(i, j) \in [m]^2$. Notice that $\beta$ is bijective and hence has an inverse $\beta^{-1} : [m^2] \rightarrow [m]^2$. By Lemma~\ref{lem: genThetaCount},
\begin{align*}
    P_{DP}(G, \mathcal{H}) = \sum_{(i, j) \in [m]^2} \prod_{k = 1}^n s_{k, (i, j)} = \sum_{\ell = 1}^{m^2} \prod_{k = 1}^n s_{k, \beta^{-1}(\ell)}.
\end{align*}
For each $k \in [n]$, let $\rho_k : [m^2] \rightarrow [m^2]$ be any permutation of $[m^2]$ such that for each $\ell_1, \ell_2 \in [m^2]$ with $\ell_1 < \ell_2$, we have $s_{k, \beta^{-1}(\rho_k(\ell_1))} \leq s_{k, \beta^{-1}(\rho_k(\ell_2))}$. Furthermore, let $x_{k, \ell} = s_{k, \beta^{-1}(\rho_k(\ell))}$ for each $\ell \in [m^2]$. By definition, for each $\ell_1, \ell_2 \in [m^2]$ with $\ell_1 < \ell_2$, we have $x_{k, \ell_1} \leq x_{k, \ell_2}$. Moreover, notice that $0 \leq n_k = x_{k, 1} = \cdots = x_{k, m} \leq x_{k, m + 1} = \cdots = x_{k, m(m - 1)} \leq x_{k, m(m - 1) + 1} = \cdots = x_{k, m^2} = n_k + 1$. By Lemma~\ref{lem: rearrangementUpperBound}, we have
\begin{align} \label{general-upper-bound-inequality}
    P_{DP}(G, \mathcal{H})
    &= \sum_{\ell = 1}^{m^2} \prod_{k = 1}^n s_{k, \beta^{-1}(\ell)} \leq \sum_{\ell = 1}^{m^2} \prod_{k = 1}^n x_{k, \ell}.
\end{align}

We now consider two cases: (1) $l_1$ is even and (2) $l_1$ is odd. First, consider case (1). If $t = 1$, then $l_1, \ldots, l_n$ are all even, and so we have $s_k = n_k + 1$, $o_k = n_k$, and
\begin{align*}
    x_{k, \ell} =
    \begin{cases}
    o_k & \text{if } \ell \in [m(m - 1)] \\
    s_k & \text{if } \ell \in [m^2] - [m(m - 1)]
    \end{cases}
\end{align*}
for each $k \in [n]$. Thus, Inequality (\ref{general-upper-bound-inequality}) becomes
\begin{align*}
    P_{DP}(G, \mathcal{H})
    &\leq \sum_{\ell = 1}^{m^2} \prod_{k = 1}^n x_{k, \ell} = \sum_{\ell = 1}^{m(m - 1)} \prod_{k = 1}^n x_{k, \ell} + \sum_{\ell = m(m - 1) + 1}^{m^2} \prod_{k = 1}^n x_{k, \ell}\\
    &= \sum_{\ell = 1}^{m(m - 1)} \prod_{k = 1}^n o_k + \sum_{\ell = m(m - 1) + 1}^{m^2} \prod_{k = 1}^n s_k
    = m(m - 1)O + mS.
\end{align*}
So, assume $t > 1$. Then, $l_1, l_{t + 1}, \ldots, l_n$ are even and $l_2, \ldots, l_t$ are odd. Hence, for each $k \in \{1, t + 1, \ldots, n\}$, we have $s_k = n_k + 1$, $o_k = n_k$, and
\begin{align*}
    x_{k, \ell} =
    \begin{cases}
    o_k & \text{if } \ell \in [m(m - 1)] \\
    s_k & \text{if } \ell \in [m^2] - [m(m - 1)].
    \end{cases}
\end{align*}
For each $k \in \{2, \ldots, t\}$, we have $s_k = n_k$, $o_k = n_k + 1$, and
\begin{align*}
    x_{k, \ell} =
    \begin{cases}
    s_k & \text{if } \ell \in [m] \\
    o_k & \text{if } \ell \in [m^2] - [m].
    \end{cases}
\end{align*}
Thus, Inequality (\ref{general-upper-bound-inequality}) becomes
\begin{align*}
    P_{DP}(G, \mathcal{H})
    &\leq \sum_{\ell = 1}^{m^2} \prod_{k = 1}^n x_{k, \ell} 
    = \sum_{\ell = 1}^m \prod_{k = 1}^n x_{k, \ell} + \sum_{\ell = m + 1}^{m(m - 1)} \prod_{k = 1}^n x_{k, \ell} + \sum_{\ell = m(m - 1) + 1}^{m^2} \prod_{k = 1}^n x_{k, \ell} \\
    &= \sum_{\ell = 1}^m \left( o_1 \prod_{k = t + 1}^n o_k \prod_{k = 2}^t s_k \right) + \sum_{\ell = m + 1}^{m(m - 1)} \prod_{k = 1}^n o_k + \sum_{\ell = m(m - 1) + 1}^{m^2} \left( s_1 \prod_{k = t + 1}^n s_k \prod_{k = 2}^t o_k \right) \\
    &= \frac{mO}{\prod_{k = 2}^t o_k} \prod_{k = 2}^t s_k + m(m - 2)O + \frac{mS}{\prod_{k = 2}^t s_k} \prod_{k = 2}^t o_k.
\end{align*}

Next, we consider case (2). If $t = 1$, then $l_1, \ldots, l_n$ are all odd, and so we have $s_k = n_k$, $o_k = n_k + 1$, and
\begin{align*}
    x_{k, \ell} =
    \begin{cases}
    s_k & \text{if } \ell \in [m] \\
    o_k & \text{if } \ell \in [m^2] - [m]
    \end{cases}
\end{align*}
for each $k \in [n]$. Then substituting these values of $x_{k, \ell}$ into Inequality (\ref{general-upper-bound-inequality}), as done in the previous case, gives us the required expression.

So, assume $t > 1$. Then, $l_1, l_{t + 1}, \ldots, l_n$ are odd and $l_2, \ldots, l_t$ are even. Hence, for each $k \in \{1, t + 1, \ldots, n\}$, we have $s_k = n_k$, $o_k = n_k + 1$, and
\begin{align*}
    x_{k, \ell} =
    \begin{cases}
    s_k & \text{if } \ell \in [m] \\
    o_k & \text{if } \ell \in [m^2] - [m].
    \end{cases}
\end{align*}
For each $k \in \{2, \ldots, t\}$, we have $s_k = n_k + 1$, $o_k = n_k$, and
\begin{align*}
    x_{k, \ell} =
    \begin{cases}
    o_k & \text{if } \ell \in [m(m - 1)] \\
    s_k & \text{if } \ell \in [m^2] - [m(m - 1)].
    \end{cases}
\end{align*}

Then substituting these values of $x_{k, \ell}$ into Inequality (\ref{general-upper-bound-inequality}), as done in the previous case, gives us the required expression.\end{proof}

Having established the appropriate upper bounds, we are now ready to complete the proof of Theorem~\ref{thm: dual}.

\begin{proof}
First, notice that if $t=1$, then $mS + m(m-1)O = P(G,m) \leq P_{DP}^*(G,m)$.  Thus, the desired result holds when $t=1$, and we may now assume that $t > 1$.  We will construct a full $m$-fold cover $\mathcal{H} = (L, H)$ of $G$ with an appropriate number of $\HH$-colorings. For each $k \in [n]$, suppose the vertices of $R_k$ written in order are $u, v_{k, 1}, \ldots, v_{k, l_k - 1}, w$. In the case $l_1 = 1$, $R_1$ has no internal vertices. Let $G' = G - \{v_{k, l_k - 1}w : k \in [n] - \{1\}\}$. Suppose $\mathcal{H}' = (L, H')$ is an $m$-fold cover of $G'$ with a canonical labeling. Begin the construction of $H$ by including all the edges in $E(H')$. The remaining edges between $L(v_{k, l_k - 1})$ and $L(w)$ for each $k \in [n] - \{1\}$ will be specified below.  Note that for each $k \in [n]$, $s_k = ((m - 1)^{l_k} + (-1)^{l_k}(m - 1)) / m$ and $o_k = ((m - 1)^{l_k} - (-1)^{l_k}) / m$.  

Let $\sigma : [m] \rightarrow [m]$ be the permutation of $[m]$ given by $\sigma(j) = (j \text{ (mod } m)) + 1$. Complete the construction of $H$ by including $\bigcup_{k \in [n] - [t]} \{(v_{k, l_k - 1}, j)(w, j) : j \in [m]\}$ and \\ $\bigcup_{k \in [t] - \{1\}} \{(v_{k, l_k - 1}, j)(w, \sigma(j)) : j \in [m]\}$ in $E(H)$. It follows from Lemma~\ref{lem: binaryThetaValues} that for each $k \in \{1, t + 1, \ldots, n\}$ we have $N(\{(u, i), (w, j)\}, \mathcal{H}_k) = ((m - 1)^{l_k} + (-1)^{l_k}(m - 1)) / m = s_k$ when $i = j$ and $N(\{(u, i), (w, j)\}, \mathcal{H}_k) = ((m - 1)^{l_k} - (-1)^{l_k}) / m = o_k$ otherwise. Again by Lemma~\ref{lem: binaryThetaValues}, for each $k \in \{2, \ldots, t\}$ we have $N(\{(u, i), (w, j)\}, \mathcal{H}_k) = ((m - 1)^{l_k} + (-1)^{l_k}(m - 1)) / m = s_k$ when $j = \sigma(i)$ and $N(\{(u, i), (w, j)\}, \mathcal{H}_k) = ((m - 1)^{l_k} - (-1)^{l_k}) / m = o_k$ otherwise. Thus, by Lemma~\ref{lem: genThetaCount},
\begin{align*}
    P_{DP}(G, \mathcal{H})
    &= \sum_{(i, j) \in [m]^2} \prod_{k = 1}^n N(\{(u, i), (w, j)\}, \mathcal{H}_k) \\
    &= \sum_{i = 1}^m \prod_{k = 1}^n N(\{(u, i), (w, i)\}, \mathcal{H}_k) + \sum_{i = 1}^m \prod_{k = 1}^n N(\{(u, i), (w, \sigma(i))\}, \mathcal{H}_k) \\
    & + \sum_{\substack{i \in [m] \\ j \in [m] - \{i, \sigma(i)\}}} \prod_{k = 1}^n N(\{(u, i), (w, j)\}, \mathcal{H}_k) \\
    &= \sum_{i = 1}^m \left[ \prod_{k \in \{1, t + 1, \ldots, n\}} N(\{(u, i), (w, i)\}, \mathcal{H}_k) \prod_{k \in \{2, \ldots, t\}} N(\{(u, i), (w, i)\}, \mathcal{H}_k) \right] \\
    & + \sum_{i = 1}^m \left[ \prod_{k \in \{1, t + 1, \ldots, n\}} N(\{(u, i), (w, \sigma(i))\}, \mathcal{H}_k) \prod_{k \in \{2, \ldots, t\}} N(\{(u, i), (w, \sigma(i))\}, \mathcal{H}_k) \right] \\
    & + \sum_{\substack{i \in [m] \\ j \in [m] - \{i, \sigma(i)\}}} \prod_{k = 1}^n o_k \\
    &= \sum_{i = 1}^m \left( \prod_{k \in \{1, t + 1, \ldots, n\}} s_k \prod_{k \in \{2, \ldots, t\}} o_k \right) + \sum_{i = 1}^m \left( \prod_{k \in \{1, t + 1, \ldots, n\}} o_k \prod_{k \in \{2, \ldots, t\}} s_k \right)\\
		&+ m(m - 2)O \\
    &= \frac{mS}{\prod_{k = 2}^t s_k}\prod_{k = 2}^t o_k + \frac{mO}{\prod_{k = 2}^t o_k} \prod_{k = 2}^t s_k + m(m - 2)O.
\end{align*}
\end{proof}

{\bf Acknowledgment.}  This paper is a combination of research projects conducted with undergraduate students: Manh Bui, Michael Maxfield, Paul Shin, and Seth Thomason at the College of Lake County during the the spring and summer of 2021. The support of the College of Lake County is gratefully acknowledged.

\appendix

\section{Proof of Lemma~\ref{lem: special}} \label{A}

In this section, we give the details of the remaining cases of this proof.

Continuing with (i), suppose $\bm{x}_1$ is odd, $\bm{x}_2$ is even, and $\bm{x}_3$ is even. Then we have the following:

\noindent\begin{minipage}{.25\linewidth}
	\begin{align*}
		x_{1, h_1(j)} =
		\begin{cases}
			n_1 \\
			n_1 + 1
		\end{cases}
	\end{align*}
\end{minipage}
\begin{minipage}{.25\linewidth}
	\begin{align*}
		, \; x_{2, h_2(j)} =
		\begin{cases}
			n_2 + 1 \\
			n_2
		\end{cases}
	\end{align*}
\end{minipage}
\begin{minipage}{.5\linewidth}
	\begin{align*}
		, \; x_{3, h_3(j)} =
		\begin{cases}
			n_3 + 1 & \;\; \text{if } j \in [m] \\
			n_3 & \;\; \text{if } j \in [m^2] - [m];
		\end{cases}
	\end{align*}
\end{minipage}

\noindent\begin{minipage}{.25\linewidth}
	\begin{align*}
		x_{1, j} =
		\begin{cases}
			n_1\\
			n_1 + 1 
		\end{cases}
	\end{align*}
\end{minipage}
\begin{minipage}{.25\linewidth}
	\begin{align*}
		, \; x_{2, f(j)} =
		\begin{cases}
			n_2 + 1 \\
			n_2 
		\end{cases}
	\end{align*}
\end{minipage}
\begin{minipage}{.5\linewidth}
	\begin{align*}
		, \; x_{3, g(j)} =
		\begin{cases}
			n_3 + 1 & \;\; \text{if } j \in [m] \\
			n_3 & \;\; \text{if } j \in [m^2] - [m].
		\end{cases}
	\end{align*}
\end{minipage}
Therefore, we obtain
\begin{align*}
	\sum_{j = 1}^{m^2} x_{1, h_1(j)} x_{2, h_2(j)} x_{3, h_3(j)}
	&= \sum_{j = 1}^m n_1(n_2 + 1)(n_3 + 1) + \sum_{j = m + 1}^{m^2} (n_1 + 1)n_2n_3 \\
	&= \sum_{j = 1}^{m^2} x_{1, j} x_{2, f(j)} x_{3, g(j)}.
\end{align*}

Continuing with (ii), we suppose $\bm{x}_1$ is odd, $\bm{x}_2$ is even, and $\bm{x}_3$ is odd. Then we have:

\noindent\begin{minipage}{.25\linewidth}
	\begin{align*}
		x_{1, h_1(j)} =
		\begin{cases}
			n_1 \\
			n_1 + 1
		\end{cases}
	\end{align*}
\end{minipage}
\begin{minipage}{.5\linewidth}
	\begin{align*}
		, \; x_{2, h_2(j)} =
		\begin{cases}
			n_2 + 1 & \;\; \text{if } j \in [m] \\
			n_2 & \;\; \text{if } j \in [m^2] - [m];
		\end{cases}
	\end{align*}
\end{minipage}
\begin{align*}
	x_{3, h_3(j)} =
	\begin{cases}
		n_3 & \text{if } j \in [2m] - [m] \\
		n_3 + 1 & \text{if } j \in [m] \cup \left( [m^2] - [2m] \right);
	\end{cases}
\end{align*}
\noindent\begin{minipage}{.25\linewidth}
	\begin{align*}
		x_{1, j} =
		\begin{cases}
			n_1 \\
			n_1 + 1
		\end{cases}
	\end{align*}
\end{minipage}
\begin{minipage}{.5\linewidth}
	\begin{align*}
		, \; x_{2, f(j)} =
		\begin{cases}
			n_2 + 1 & \;\; \text{if } j \in [m] \\
			n_2 & \;\; \text{if } j \in [m^2] - [m].
		\end{cases}
	\end{align*}
\end{minipage}
We also have $N = m$, and so
\begin{align*}
	x_{3, g(j)} =
	\begin{cases}
		n_3 + 1 & \text{if } j \in [m] \cup \left( [m^2] - [2m] \right) \\
		n_3 & \text{if } j \in [2m] - [m].
	\end{cases}
\end{align*}
Therefore, we obtain
\begin{align*}
	&\sum_{j = 1}^{m^2} x_{1, h_1(j)} x_{2, h_2(j)} x_{3, h_3(j)} \\
	&= \sum_{j = 1}^m n_1(n_2 + 1)(n_3 + 1) + \sum_{j = m + 1}^{2m} (n_1 + 1)n_2n_3 + \sum_{j = 2m + 1}^{m^2} (n_1 + 1)n_2(n_3 + 1) \\
	&= \sum_{j = 1}^{m^2} x_{1, j} x_{2, f(j)} x_{3, g(j)}.
\end{align*}

Continuing with (iii), suppose $\bm{x}_1$, $\bm{x}_2$, and $\bm{x}_3$ are all odd. Then we have the following:
\begin{align*}
	x_{1, h_1(j)} =
	\begin{cases}
		n_1 & \text{if } j \in [m] \\
		n_1 + 1 & \text{if } j \in [m^2] - [m],
	\end{cases}
\end{align*}
\begin{align*}
	x_{2, h_2(j)} =
	\begin{cases}
		n_2 & \text{if } j \in [2m] - [m] \\
		n_2 + 1 & \text{if } j \in [m] \cup \left( [m^2] - [2m] \right),
	\end{cases}
\end{align*}
\begin{align*}
	x_{3, h_3(j)} =
	\begin{cases}
		n_3 & \text{if } j \in [3m] - [2m] \\
		n_3 + 1 & \text{if } j \in [2m] \cup \left( [m^2] - [3m] \right),
	\end{cases}
\end{align*}
\begin{align*}
	x_{1, j} =
	\begin{cases}
		n_1 & \text{if } j \in [m] \\
		n_1 + 1 & \text{if } j \in [m^2] - [m],
	\end{cases}
\end{align*}
and
\begin{align*}
	x_{2, f(j)} =
	\begin{cases}
		n_2 + 1 & \text{if } j \in [m(m - 1)] \\
		n_2 & \text{if } j \in [m^2] - [m(m - 1)].
	\end{cases}
\end{align*}
We also have $N = m$, and so
\begin{align*}
	x_{3, g(j)} =
	\begin{cases}
		n_3 + 1 & \text{if } j \in [m] \cup \left( [m^2] - [2m] \right) \\
		n_3 & \text{if } j \in [2m] - [m].
	\end{cases}
\end{align*}
Therefore, we obtain
\begin{align*}
	&\sum_{j = 1}^{m^2} x_{1, h_1(j)} x_{2, h_2(j)} x_{3, h_3(j)} \\
	&= \sum_{j = 1}^m n_1(n_2 + 1)(n_3 + 1) + \sum_{j = m + 1}^{2m} (n_1 + 1)n_2(n_3 + 1) + \sum_{j = 2m + 1}^{3m} (n_1 + 1)(n_2 + 1)n_3 \\
	&\phantom{=} + \sum_{j = 3m + 1}^{m^2} (n_1 + 1)(n_2 + 1)(n_3 + 1) \\
	&= mn_1(n_2 + 1)(n_3 + 1) + m(n_1 + 1)n_2(n_3 + 1) + m(n_1 + 1)(n_2 + 1)n_3 \\
	&\phantom{=} + m(m - 3)(n_1 + 1)(n_2 + 1)(n_3 + 1) \\
	&= \sum_{j = 1}^m n_1(n_2 + 1)(n_3 + 1) + \sum_{j = m + 1}^{2m} (n_1 + 1)(n_2 + 1)n_3 \\
	&\phantom{=} + \sum_{j = 2m + 1}^{m(m - 1)} (n_1 + 1)(n_2 + 1)(n_3 + 1) + \sum_{j = m(m - 1) + 1}^{m^2} (n_1 + 1)n_2(n_3 + 1) \\
	&= \sum_{j = 1}^{m^2} x_{1, j} x_{2, f(j)} x_{3, g(j)}.
\end{align*}

\section{Proof of Lemma~\ref{lem: theta-dp-lower-bound}} \label{B}

To complete the proof of the claim for each $k \in [3]$,
\begin{align*}
	s_k = \frac{(m - 1)^{l_k} - (-1)^{l_k}}{m} + (-1)^{l_k} = \frac{(m - 1)^{l_k} + (-1)^{l_k}(m - 1)}{m}
\end{align*}
and
\begin{align*}
	o_k = \frac{(m - 1)^{l_k} + (-1)^{l_k + 1}}{m}, 
\end{align*}  
suppose that $l_k$ is odd. Then we have
\begin{align*}
	s_{k, (i, j)} = \frac{(m - 1)^{l_k} - (-1)^{l_k}}{m} - 1
\end{align*}
for exactly $m$ choices of $(i, j) \in [m]^2$, and
\begin{align*}
	s_{k, (i, j)} = \frac{(m - 1)^{l_k} - (-1)^{l_k}}{m}
\end{align*}
for the remaining $m(m - 1)$ choices of $(i, j) \in [m]^2$. Therefore, we have
\begin{align*}
	n_k = \frac{(m - 1)^{l_k} - (-1)^{l_k}}{m} - 1
\end{align*}
and
\begin{align*}
	x_{k, \ell} =
	\begin{cases}
		n_k & \text{if } \ell \in [m] \\
		n_k + 1 & \text{if } \ell \in [m^2] - [m].
	\end{cases}
\end{align*}
It follows by definition that $\bm{x}_k$ is odd, so that
\begin{align*}
	s_k = n_k = \frac{(m - 1)^{l_k} - (-1)^{l_k}}{m} + (-1)^{l_k}
\end{align*}
and
\begin{align*}
	o_k = n_k + 1 = \frac{(m - 1)^{l_k} - (-1)^{l_k}}{m}.
\end{align*}

\end{document}